\documentclass[12pt,twoside]{article}
\usepackage{graphicx}
\usepackage{tikz}
\usepackage{tikz-cd}
\usepackage{amsfonts}
\usepackage{amsbsy}
\usepackage{amssymb}
\usepackage{float}
\usepackage{amsmath,amsthm}
\usepackage{verbatim}
\usepackage{amsfonts}
\usepackage{pst-all}
\usepackage{pstricks}
\usepackage[T1]{fontenc}
\usepackage{multicol}
\usepackage{color}
\usepackage[english]{babel}
\usepackage[autostyle, english = american]{csquotes}
\MakeOuterQuote{"}
\usepackage{lipsum}
\usepackage[document]{ragged2e}
\usepackage{pgf}
\usepackage{mathtools}
\usepackage{genyoungtabtikz}
\usepackage[onehalfspacing]{setspace}
\usepackage[parfill]{parskip}
\usepackage{hyperref}
\hypersetup{
    colorlinks=true,
    linkcolor=blue,
    filecolor=blue,      
    urlcolor=blue,
    citecolor=blue,
    pdftitle={Overleaf Example},
    pdfpagemode=FullScreen,
    }

\DeclareFontFamily{U}{matha}{\hyphenchar\font45}
\DeclareFontShape{U}{matha}{m}{n}{
      <5> <6> <7> <8> <9> <10> gen * matha
      <10.95> matha10 <12> <14.4> <17.28> <20.74> <24.88> matha12
      }{}
\DeclareSymbolFont{matha}{U}{matha}{m}{n}
\DeclareMathSymbol{\wedge}         {2}{matha}{"5E}
\DeclareMathSymbol{\vee}           {2}{matha}{"5F}

\newtheorem{thm}{Theorem}
\newtheorem{prop}{Proposition}[section]
\newtheorem{cor}[thm]{Corollary}
\newtheorem{defn}[thm]{Definition}

\newtheorem{lem}[thm]{Lemma}

\newtheorem{rem}[prop]{Remark}
\newtheorem{exm}[thm]{Example}

\newcommand{\Z}{\mathbb{Z}}

\makeatletter
\def\blfootnote{\xdef\@thefnmark{}\@footnotetext}
\makeatother

\begin{document}
\title{Degenerations and order of graphs realized by finite abelian groups}
\author{{Rameez Raja
\footnote{Department of Mathematics, National Institute of Technology Srinagar-190006, Jammu and Kashmir, India.  Email: rameeznaqash@nitsri.ac.in}}}

\date{}

\maketitle
\vskip 5mm
\noindent{\footnotesize \bf Abstract.}
Let $G_1$ and $G_2$ be two groups. If a group homomorphism $\varphi:G_1 \rightarrow G_2$ maps $a\in G_1$ into $b\in G_2$ such that $\varphi(a) = b$, then we say $a$ \textit{degenerates} to $b$ and if every element of $G_1$ degenerates to elements in $G_2$, then we say $G_1$ degenerates to $G_2$. We discuss degeneration in graphs and show that degeneration in groups is a particular case of degeneration in graphs. We exhibit some interesting properties of degeneration in graphs. We use this concept to present a pictorial representation of graphs realized by finite abelian groups. We discus some partial orders on the set $\mathcal{T}_{p_1 \cdots p_n}$ of all graphs realized by finite abelian $p_r$-groups, where each $p_r$, $1\leq r\leq n$, is a prime number. We show that each finite abelian $p_r$-group of rank $n$ can be identified with \textit{saturated chains} of \textit{Young diagrams} in the poset $\mathcal{T}_{p_1 \cdots p_n}$. We present a combinatorial formula which represents the degree of a projective representation of a symmetric group. This formula determines the number of different \textit{saturated chains} in $\mathcal{T}_{p_1 \cdots p_n}$ and the number of finite abelian groups of different orders.
\vskip 3mm



\section{\bf Introduction}

A notion of degeneration in groups was introduced in \cite{KA} to parametrize the orbits in a finite abelian group under its full automorphism group by a finite distributive lattice. The authors in \cite{KA} were motivated by attempts to understand the decomposition of the weil representation associated to a finite abelian group $G$. Note that the sum of squares of the multiplicities in the Weil representation is the number of orbits in $G \times \hat{G}$ under automorphisms of a symplectic bicharacter, where $\hat{G}$ denotes the Pontryagin dual of $G$.\\ 

The above combinatorial description is one of the explorations between  groups and combinatorial structures  (posets and lattices). There is an intimate relationship between between groups and other combinatorial structures (graphs). For example, any graph $\Gamma$ give rise to its automorphism group whereas any group with its generating set give rise to a realization of a group as a graph (Cayley graph).

Recently, authors in \cite{ER} studied the \textit{group-annihilator} graph $\Gamma(G)$ realized by a finite abelian group $G$ (viewed as a $\mathbb{Z}$-module) of different ranks. The vertices of $\Gamma(G)$ are all elements of $G$ and two vertices $x ,y \in G$ are adjacent in $\Gamma(G)$ if and only if $[x : G][y : G]G = \{0\}$, where $[x : G] = \{r\in\mathbb{Z} : rG \subseteq \mathbb{Z}x\}$ is an ideal of a ring $\mathbb{Z}$. They investigated the concept of creation sequences in $\Gamma(G)$ and determined the multiplicities of eigenvalues $0$ and $-1$ of $\Gamma(G)$. Interestingly, they considered orbits of the symmetric group action: $Aut(\Gamma(G)) \times G \longrightarrow G$ and proved that the representatives of orbits are the Laplacian eigenvalues of $\Gamma(G)$. 

There are number of realizations of groups as graphs. The  generating graph \cite{LS} realized by a simple group was introduced to get an insight that might ultimately guide us to a new proof of the classification of simple groups. The graphs such as power graph \cite{CS}, intersection graph \cite{ZB} and the commuting graph \cite{BF} were introduced to study the information contained in the graph about the group.

Moreover, the realizations of rings as graphs were introduced in \cite{AL, Bk}. The aim of considering these realizations of rings as graphs is to study the interplay between combinatorial and ring theoretic properties of a ring $R$. This concept was further studied in \cite{FL, M, SR, RR} and was extended to modules over commutative rings in \cite{SR1}.

The main objective of this work is to investigate some deeper interconnections between partitions of a number, young diagrams, finite abelain groups, group homomorphisms, graph homomorphisms, posets and lattices. This investigation will lead us to develop a theory which is going to simplify the concept of degeneration of elements in groups and also provide a lattice of finite abelian groups in which each \textit{saturated chain} of length $n$ can be identified with a finite abelian $p_r$-group of rank $n$. 

This research article is organized as follows. In section 2, we discuss some results related to degeneration in groups and group-annihilator graphs realized by finite abelian groups. Section 3 is dedicated to the study of degenerations in graphs realized by finite abelian groups. We present a pictorial sketch which illustrates degeneration in graphs. Finally in section 4, we investigate multiple relations on the set $\mathcal{T}_{p_1 \cdots p_n}$ and furnish the information contained in a locally finite distributive lattice about finite abelian groups. We provide a combinatorial formula which represents degree of a projective representation of a symmetric group and the number of \textit{saturated chains} from empty set to some non-trivial member of $\mathcal{T}_{p_1 \cdots p_n}$.

 \section{\bf Preliminaries}

Let $\lambda = (\lambda_1, \lambda_2, \cdots, \lambda_r)$ be a partition of $n$ denoted by $\lambda \vdash n$, where $n\in \Z_{>0}$ is a positive integer. For any  $\mu \vdash n$, we have an abelian group of order $p^n$ and conversely every abelian group corresponds to some partition of $n$. In fact, if $H_{\mu, p} =  \mathbb{Z}/p^{{\mu}_1}\mathbb{Z} ~\oplus~ \mathbb{Z}/p^{{\mu}_2}\mathbb{Z} ~\oplus~ \cdots ~\oplus~ \mathbb{Z}/p^{{\mu}_r}\mathbb{Z}$ is a subgroup of $G_{\lambda, p}$ ($G_{\lambda, p} = \mathbb{Z}/p^{\lambda_1}\mathbb{Z} \oplus \mathbb{Z}/p^{\lambda_2}\mathbb{Z} \oplus \cdots \oplus \mathbb{Z}/p^{\lambda_r}\mathbb{Z}$ is a finite abelian $p$-group), then $\mu_1 \leq \lambda_1, \mu_2 \leq \lambda_2, \cdots, \mu_r \leq  \lambda_r$. If these inequalities holds we write $\mu \subset \lambda$, that is a \textquotedblleft containment order\textquotedblright on partitions. For example, a $p$-group $\mathbb{Z}/p^{7}\mathbb{Z} ~\oplus~ \mathbb{Z}/p\mathbb{Z} ~\oplus~ \mathbb{Z}/p\mathbb{Z}$ is of type $\lambda = (7, 1, 1)$. The possible types for its subgroup are: $(7, 1, 1), (6, 1, 1), (5, 1, 1), (4, 1, 1)$, 
\noindent$(3, 1, 1), (2, 1, 1), (1, 1, 1), 2(7, 1), 2(6, 1), 2(5, 1), 2(4, 1), 2(3,1), 2(2, 1), 2(1, 1), (7), (6), (5), (4)$, 
$\noindent (3), (2), 2(1)$. 

Note that the types $(7, 1), (6, 1), (5, 1), (4, 1), (3,1), (2, 1), (1, 1)$ are appearing twice in the sequence of partitions for a subgroup.

The authors in \cite{KA} have considered the group action: $Aut(G) \times G \rightarrow G$, where $Aut(G)$ is an automorphism group of $G$ and studied $Aut(G)\setminus G$, the set of all disjoint $Aut(G)$-orbits in $G$. The group $\mathbb{Z}/p^{k}\mathbb{Z}$ has $k$ orbits of non-zero elements under the action of its automorphism group, represented by elements $1, p, \cdots, p^{k-1}$. We denote orbits of the group action: $Aut(\mathbb{Z}/p^{k}\mathbb{Z}) \times \mathbb{Z}/p^{k}\mathbb{Z}\longrightarrow \mathbb{Z}/p^{k}\mathbb{Z}$ by $\mathcal{O}_{k, p^{m}}$, where $0\leq m \leq k-1$.

Miller \cite{ML}, Schwachh\"ofer and Stroppel \cite{SS} provided some well known formulae for the cardinality of the set $Aut(G_{\lambda, p})$ $\setminus$ $G_{\lambda, p}$ . 

\begin{defn}\label{df1}
(Degeneration in groups) \cite{KA}. Let $G_1$ and $G_2$ be two groups, then  $a \in G_1$ degenerates to $b \in G_2$, if a homomorphism $\varphi : G_1 \longrightarrow G_2$ maps $a$ into $b$ such that $\varphi(a) = b$. 
\end{defn}

The following result provide a characterization for degenerations of elements of the group $\mathbb{Z}/p^{k}\mathbb{Z}$ to elements of the group $\mathbb{Z}/p^{l}\mathbb{Z}$, where $k \leq l$.

\begin{lem}\label{lm1}\cite{KA}.
$p^{r}u \in \mathcal{O}_{k, p^{r}}$ in $\mathbb{Z}/p^{k}\mathbb{Z}$ degenerates to $p^{s}v \in \mathcal{O}_{l, p^{s}}$ in  $\mathbb{Z}/p^{l}\mathbb{Z}$ if and only if $r \leq s$ and $k-r \geq l-s$, where $u, v$ are relatively prime to $p$, $r<k$ and $s<l$. If in addition $p^{s}v \in \mathcal{O}_{l, p^{s}}$ degenerates to $p^{r}u \in \mathcal{O}_{k, p^{r}}$, then $k = l$ and $r = s$.
\end{lem}

By Lemma \ref{lm1}, it is easy to verify that degeneracy is a partial order relation on the set of all orbits of non-zero elements in $\mathbb{Z}/p^{k}\mathbb{Z}$. The diagrammatic representation (Hasse diagram) of the set $Aut(\mathbb{Z}/p^{k}\mathbb{Z})\setminus \mathbb{Z}/p^{k}\mathbb{Z}$ with respect to degeneracy, which is called a fundamental poset is presented in [Figure 1 \cite{KA}]. 

Let $a = (a_1, a_2, \cdots, a_r) \in G_{\lambda, p}$, the \textit{ideal of a} in $Aut(G_{\lambda, p})$ $\setminus$ $G_{\lambda, p}$ denoted by $I(a)$ is the ideal generated by orbits of non-zero coordinates $a_i\in \mathbb{Z}/p^{\lambda_{i}}\mathbb{Z}$. One of the explorations between ideals of posets, partitions and orbits of finite abelian groups is the following interesting result.

\begin{thm}\label{thm1}\cite{KA}.
Let $\lambda$ and $\mu$ be any two given partitions and $a\in G_{\lambda, p}$, $b\in G_{\mu, p}$. Then $a$ degenerates to $b$ if and only if $I(b) \subset I(a)$.  
\end{thm}

The enumeration of orbits as ideals, first as counting ideals in terms of their boundaries, and the second as counting them in terms of anti chains of maximal elements is presented in [Example 6.1, 6.2 \cite{KA}].

Please see sections 7 and 8 of \cite{KA} for results related to embedding of the lattice of orbits of $G_{\lambda, p}$ into the lattice of characteristic subgroups of $G_{\lambda, p}$, formula for the order of the characteristic subgroup associated to an orbit, computation of a monic polynomial in $p$ (with integer coefficients) using mobius inversion formula representing cardinality of the orbit in $G_{\lambda, p}$.

Let $\Gamma = (V, E)$ be a simple connected graph and let $\Gamma_1$ and $\Gamma_2$ be two simple connected graphs, recall a mapping $\phi: V(\Gamma_1) \rightarrow V(\Gamma_2)$ is a \textit{homomorphism} if it preserves edges, that is, for any edge $(u, v)$ of $\Gamma_1$, $(\phi(u), \phi(v))$ is an edge of $\Gamma_2$, where $u, v\in V(\Gamma_1)$. A homomorphism  $\phi: V(\Gamma_1) \rightarrow V(\Gamma_2)$ is faithful when there is an edge between two pre images $\phi^{-1}(u)$ and $\phi^{-1}(u)$ such that $(u, v)$ is an edge of $\Gamma_2$, a faithful bijective homomorphism is an \textit{isomorphism} and in this case we write $\Gamma_1 \cong \Gamma_2$. An isomorphism from $\Gamma$ to itself is an \textit{automorphism} of $\Gamma$, it is well known that set of automorphisms of $\Gamma$ forms a group under composition, we denote the group of automorphisms of $\Gamma$ by $Aut(\Gamma)$. Understanding the automorphism group of a graph is a guiding principle for understanding objects by their symmetries. 

Consider the group action: $Aut(\Gamma) ~acting~on ~V(\Gamma)$ by some permutation of $Aut(\Gamma)$, that is, 

\hskip .9cm \hskip .9cm \hskip .9cm \hskip .9cm \hskip .9cm \hskip .9cm $Aut(\Gamma) \times V(\Gamma) \rightarrow V(\Gamma)$, 

\hskip .9cm \hskip .9cm \hskip .9cm \hskip .9cm \hskip .9cm \hskip .9cm \hskip .9cm \hskip .3cm $\sigma(v) = u$,

\noindent where $\sigma \in Aut(\Gamma)$ and $v, u\in V(\Gamma)$ are any two vertices of $\Gamma$. This group action is called a \textit{symmetric action} \cite{ER}.

Consider a finite abelian non-trivial group $G$ with identity element $0$ and view $G$ as a $\mathbb{Z}$-module. For $a\in G$, set $[a : G] =\{x\in \mathbb{Z} ~|~ xG\subseteq \mathbb{Z}a\}$, which clearly is an ideal of $\mathbb{Z}.$ For $a \in G$, $G/\mathbb{Z}a$ is a $\mathbb{Z}$-module. So,  $[a : G]$ is a annihilator of $G/\mathbb{Z}a$, $[a:G]$ is called a $a$-annihilator of $G.$ Also, an element $a$ is called an \textit{ideal-annihilator} of $G$ if there exists a non-zero element $b$ of $G$ such that $[a : G][b : G]G = \{0\}$, where $[a : G][b : G]$ denotes the product of ideals of $\mathbb{Z}$. The element $0$ is a trivial ideal-annihilator of $G$, since $[0 : G][b : G]G = ann(G)[b : G]G = \{0\}$, $ann(G)$ is an annihilator of $G$ in $\mathbb{Z}$.

Given an abelian group $G$, the \textit{group-annihilator} graph is defined to be the graph $\Gamma(G) = (V(\Gamma(G))$, $E(\Gamma(G)))$ with vertex set $V(\Gamma(G))= G$ and for two distinct $a, b\in V(\Gamma(G))$, the vertices $a$ and $b$ are adjacent in $\Gamma(G)$ if and only if $[a : G][b : G]G = \{0\}$, that is, $E(G) = \{(a, b)\in G \times G : [a : G][b : G]G = \{0\}\}$. 

For a cyclic group $G = \mathbb{Z}/p^{n}\Z$ ($n\geq1$), it is easy to verify that the orbits of the action: $Aut(G) \times G \longrightarrow G$ are same as the orbits of the symmetric action:   $Aut(\Gamma(G)) \times G \longrightarrow G$ which are given as follows,

\begin{center}
$\mathcal{O}_{n, p^i} = \{p^i\alpha (mod~p^{n})\mid \alpha \in \Z, (\alpha, p) =1\}$,\end{center}

where $i \in [0, n]$. Furthermore, for $0\leq i< j \leq n$,  $p^i\alpha \equiv p^j \alpha' (mod~p^{n})\text{ where } (\alpha, p) = 1 \text{ and } (\alpha', p) = 1$. Consequently, we have for $i \neq j$, $\mathcal{O}_{n, p^i} \cap  \mathcal{O}_{n, p^j} = \emptyset$.

Any element $a \in \Z/p^{n}\Z$ can be expressed as,

\begin{center}
$a \equiv p^{n-1}b_1 + p^{n-2} b_2 +\cdots + p b_{n-1 } +b_{n} (mod~p^{n})$,\end{center}

where $b_i \in [1, p-1]$. If $a \in \mathcal{O}_{n, 1}$, then $b_{n} \neq 0.$ So, $|\mathcal{O}_{n, 1}| = p^{n-1} (p-1) = \phi (p^{n})$. If $a' \in \mathcal{O}_{n, p}$, then for some $a \in \mathcal{O}_{n, 1}$ $a' = pa$, that is, $b_{n}\neq 0$, so $|\mathcal{O}_{n, p}|
= \frac{\phi (p^{n})}{p}$. Similarly, for $i \in [0, n]$, we have
  $|\mathcal{O}_{n ,p^i}| = \frac{\phi(p^{n})}{p^i}$.
  
\begin{prop}\cite{ER}. Let $G =\Z/p^{n}\Z$ be a cyclic group of order $p^{n}$, where $n \ge 2$. Then for each $a\in \mathcal{O}_{n, p^i}$ with $i\in [1, n]$, the $a-$annihilator of $G$ is $[a: G] = p^i \Z $
\end{prop} 

Thus if we consider the symmetric group action: $Aut(\Gamma(G)) \times G \longrightarrow G$, then for $G =\Z/p^{n}\Z$, the group-annihilator graph realized by $G$ is defined as $\Gamma (G) = (V(\Gamma(G)), E(\Gamma(G)))$,
where $V(\Gamma(G)) = \Z/p^{n}\Z$ and two vertices $u \in \mathcal{O}_{n, p^i}$, $v \in \mathcal{O}_{n, p^j}$ are adjacent in $\Gamma(G)$ if and only if $ i + j \geq n$.

Therefore, from the above observation it follows that the vertices of the graph $\Gamma(G)$ are parametrized by representatives of orbits of the group action: $Aut(\Gamma(G)) \times G \longrightarrow G$. Thus an element $0\in \mathcal{O}_{n, p^{n}}$ of $G$ is adjacent to all vertices in $\Gamma (G)$, elements $a \in \mathcal{O}_{n, 1}$ which are prime to order of $G$ are adjacent to $0$ only in $\Gamma (G)$. Furthermore, elements of the orbit $\mathcal{O}_{n, p}$ are adjacent to $0$ and elements of the orbit $\mathcal{O}_{n, p^{n-1}}$, elements of the orbit $\mathcal{O}_{n, p^2}$ are adjacent to $0$ and elements of the orbits $\mathcal{O}_{n, p^{n-1}}$,  $\mathcal{O}_{n, p^{n-2}}$. Thus, for $k \geq 1$, elements of the orbit
 $\mathcal{O}_{n, p^k}$ are adjacent to elements of the orbits $\mathcal{O}_{n, p^{n-k}}$,  $\mathcal{O}_{n, p^{n-k + 1}}, \cdots, \mathcal{O}_{n, p^{n -1}}$, $\mathcal{O}_{n, p^{n}}$.
 
\begin{thm} \cite{ER}. Let $n$ be a positive integer. Then for the $p$-group $G = (\Z/p^{n}\Z)^{\ell}  $ of rank $\ell\geq2,$ and $(a_1,\ldots,a_l)\in G$, the $(a_1,\ldots,a_l)$-annihilator of $G$ is $p^n\Z.$ In particular the corresponding group-annihilator graph realized by $G$ is a complete graph. 
\end{thm}
 
Note that the action of $Aut(\Gamma((\mathbb{Z}/p\mathbb{Z})^{\ell}))$ on $(\mathbb{Z}/p\mathbb{Z})^{\ell}$ is transitive, since an automorphism of $\Gamma((\mathbb{Z}/p\mathbb{Z})^{\ell})$ map any vertex to any other vertex and this does not place any restriction on where any of the other $p^{\ell} - 1$ vertices are mapped, as they are all mutually connected in $\Gamma((\mathbb{Z}/p\mathbb{Z})^{\ell})$. This implies $Aut(\Gamma((\mathbb{Z}/p\mathbb{Z})^{\ell})) \setminus(\mathbb{Z}/p\mathbb{Z})^{\ell}$ is a single orbit of order $p^{\ell}$.

For more information regarding $a-$annihilators, $(a, b)-$annihilators and $(a_1, a_2, \cdots, a_l)$
$-$annihilators of finite abelian $p$-groups, please see section 3 of \cite{ER}.  

We conclude this section by an example which illustrates the parametrization of vertices of the group-annihilator graph $\Gamma(G)$ by representatives of orbits of the symmetric action on $G$.

\begin{exm} Let $G = \mathbb{Z}/2^4\mathbb{Z}$ be a finite abelian. Consider the group action: $Aut(\Gamma(G)) \times G \longrightarrow G$. The orbits of this action are: $\mathcal{O}_{4, 2^4} = \{0\}$, $\mathcal{O}_{4, 1} = \{1, 3, 5, 7\}$, $\mathcal{O}_{4, 2} = \{2, 6, 10, 14\} = \{2a ~|~ (a, 2) = 1\}$, $\mathcal{O}_{4, 2^2} = \{4, 12\} = \{2^2a ~|~ (a, 2) = 1\} $ and $\mathcal{O}_{4, 2^3} = \{8\} = \{2^3a ~|~ (a, 2) = 1\}$. Note that orbits of elements $3, 5, 7$ are same as the orbit of $1$, orbits of $6, 10, 14$ are same as the orbit of $2$ and orbit of $12$ is same as the orbit of $4$.  Therefore, the group $G$ has $4$ orbits of nonzero elements under the action of $Aut(\Gamma(G))$ represented by $1, 2, 2^2, 2^3$. The group-annihilator graph realized by $G$ with its orbits is shown in Figure \eqref{1}.\end{exm} 

\begin{figure}[H]
 \begin{center}
 \includegraphics[scale=.415]{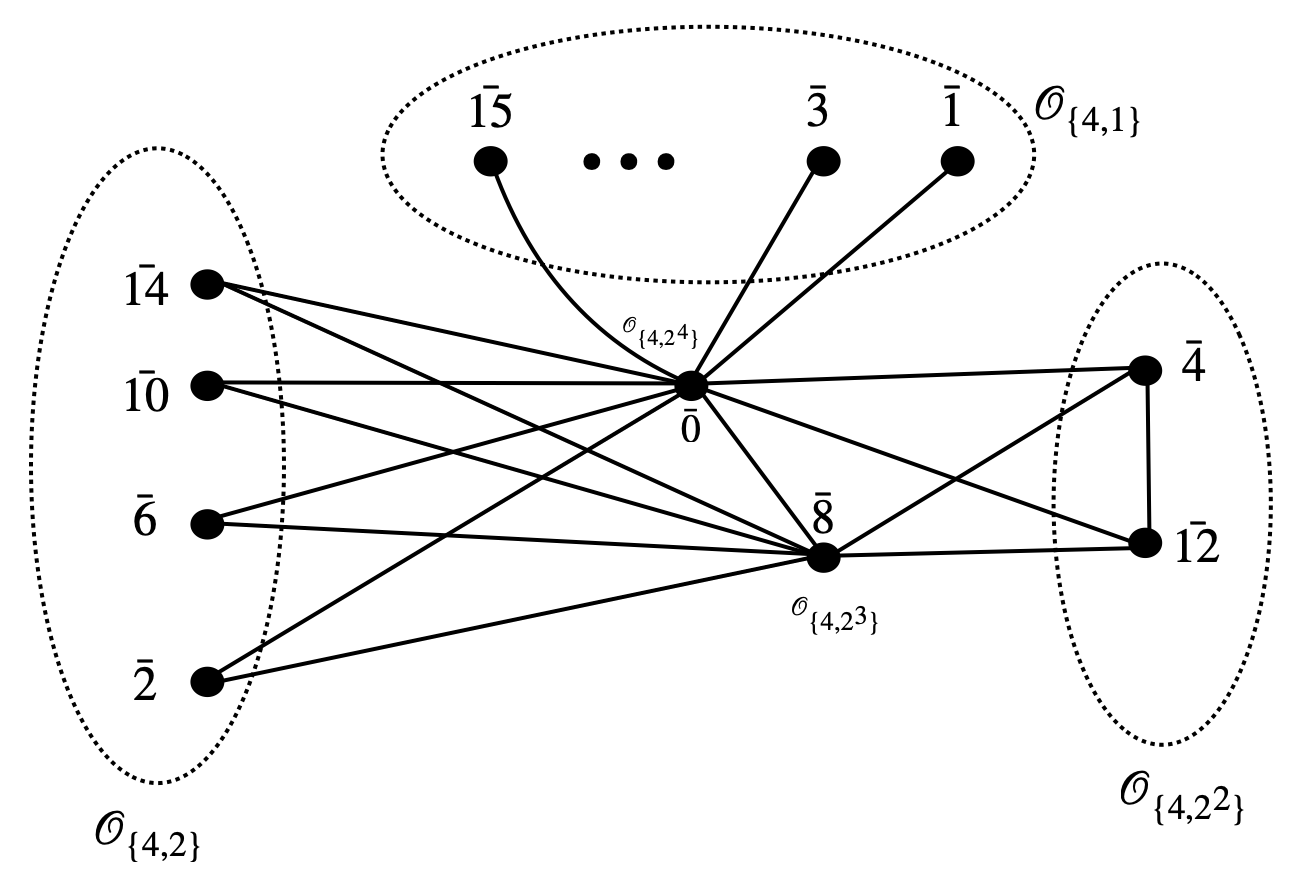}
 \end{center}
\caption{$\Gamma(\Z/2^{4}\Z)$ with its orbits}
\label{1}
 \end{figure}

\section{Degeneration in graphs}
This section is devoted to the study of degeneration in graphs. We show that every group homomorphism is a graph homomorphism. We employ the methods of degeneration in graphs to simply the techniques used to establish degenerations of elements in finite abelian groups \cite{KA}. 

As far as groups are concerned, there are always homomorphisms (trivial homomorphisms) from one group to another. Any \textit{source group} (a group where from we have the map) can be mapped by a homomorphism into \textit{target group} (a group where the elements are mapped) by simply sending all of its elements to
the identity of the target group. In fact, the study of kernels is very important in algebraic structures. In the context of simple graphs, the notion of a homomorphism is far more restrictive. Indeed, there need not be a homomorphism between two graphs, and these cases are as much a part of the theory as those where homomorphisms do exist. There are other categories where homomorphisms do not always exist between two objects, for example, the category of bounded lattices or that of semi-groups.

The answer to the question that \enquote{every group homomorphism is a graph homomorphism} is affirmative, and the same is discussed in the following result. Note that the orbits of elements of actions (automorphism group and symmetric) on finite abelian $p$-group of rank one coincide and it can be explored further on abelian $p$-groups of different ranks.

\begin{prop}\label{prp1}
Every group homomorphism which maps elements from orbits $\mathcal{O}_{k,p^i}$ to orbits $\mathcal{O}_{l,p^j}$ is a graph homomorphism, where $1\leq i\leq k$, $1\leq j\leq l$ and $k \leq l$. 
\end{prop}
\begin{proof}
The group homomorphisms are uniquely determined by the image of unity element in the target group and order of the element divides order of unity in the source group. Let $\tau(a)$ be the image of unity in the target group. Therefore, we have $\tau(a) = a_1, a_2, \cdots, a_{p^k}$, where $a_1, a_p, \cdots, a_{p^{k}}$ are elements of orbits, $\mathcal{O}_{l,1}$, $\mathcal{O}_{l,p}$, $\cdots$, $\mathcal{O}_{l,p^k}$. Note that $k\leq l$, therefore we have the following inequalities concerning the cardinalities of obits,
\begin{center}
$|\mathcal{O}_{k,1}| \leq |\mathcal{O}_{l,1}|$,

$|\mathcal{O}_{k,p}| \leq |\mathcal{O}_{l,p}|$,

\hskip .1cm \vdots

$|\mathcal{O}_{k,p^k}| \leq |\mathcal{O}_{l,p^k}|$.
\end{center} 
If $\tau(a)\in \mathcal{O}_{l, 1}$, then under the monomorphism the elements of orbits are mapped as,

\begin{center}
$ \mathcal{O}_{k,1} \xhookrightarrow{~1-1~} \mathcal{O}_{l,1}$,

$ \mathcal{O}_{k,p} \xhookrightarrow{~1-1~} \mathcal{O}_{l,p}$,

\hskip .1cm \vdots

$ \mathcal{O}_{k,p^{k-1}} \xhookrightarrow{~1-1~} \mathcal{O}_{l,p^{k-1}}$,

$\mathcal{O}_{k,p^{k}} \xhookrightarrow{~1-1~} \mathcal{O}_{l,p^{l}}$.
\end{center}

If $\tau(a)\in \mathcal{O}_{l,p}$, then  elements of orbits are mapped as,

\begin{center}
$ \mathcal{O}_{k,1} \twoheadrightarrow  \mathcal{O}_{l,p}$,

$ \mathcal{O}_{k,p} \twoheadrightarrow \mathcal{O}_{l,p^2}$,

\hskip .1cm \vdots

$ \mathcal{O}_{k,p^{k-1}} \twoheadrightarrow \mathcal{O}_{l,p^{k}}$,

$\mathcal{O}_{k,p^{k}} \twoheadrightarrow \mathcal{O}_{l,p^{l}}$.
\end{center}
 
Thus it follows that if $\tau(a)\in \mathcal{O}_{l,p^t}$ for $(0\leq t\leq k-1)$, then every element of the orbit $\mathcal{O}_{k,p^t}$  is mapped to elements of the orbit $\mathcal{O}_{l,p^{t+1}}$.

Under the symmetric action the orbits of vertices are same as the orbits listed above. Note that the vertices of the orbit $\mathcal{O}_{k,1}$ are only adjacent to the vertex in $\mathcal{O}_{k,p^k}$, vertices of the orbit $\mathcal{O}_{k,p}$ are adjacent to vertices in $\mathcal{O}_{k,p^k}$ and $\mathcal{O}_{k,p^{k-1}}$ and so on. Thus if $\tau(a)\in \mathcal{O}_{l,1}$, then for $0\leq i\leq j\leq k$, every edge $(u, v) \in  \mathcal{O}_{k,p^i} \times \mathcal{O}_{k,p^j}$ is mapped to edges $(\tau(u), \tau(v)) \in  \mathcal{O}_{l,p^r} \times \mathcal{O}_{l,p^s}$, where $0\leq r\leq s\leq l$. Therefore $\tau$ is a graph homomorphism. Similarly it can be verified that all other group homomorphisms are graph homomorphisms, since the adjacencies are preserved under all group homomorphisms.
\end{proof}

\begin{rem}
The converse of the preceding result is not true, that is, a graph homomorphism between two graphs realised by some groups need not to be a group homomorphism. To illustrate this we consider the \enquote{distribution of edges in orbits}. Theoretically,  distribution of edges is carried out in a way that for sufficiently large $l$, a graph homomorphism is acting on vertices in orbits $\mathcal{O}_{k,p^k}$, $\mathcal{O}_{k,1}$ such that $\mathcal{O}_{k,p^{k}} \xhookrightarrow{identity} \mathcal{O}_{l,p^{l}}$, $\mathcal{O}_{k,p^{k-1}} \xhookrightarrow{identity} \mathcal{O}_{l,p^{k-1}}$, $\cdots$, $\mathcal{O}_{k,p} \xhookrightarrow{identity} \mathcal{O}_{l,p}$. Some vertices of $\mathcal{O}_{k,1}$ are mapped to itself in $\mathcal{O}_{l,1}$ whereas the remaining are mapped to vertices in $\mathcal{O}_{l,p}$. So, under the above distribution some edges in $\mathcal{O}_{k,p^{k}} \times \mathcal{O}_{k,1}$ are mapped to edges in $\mathcal{O}_{l,p^{l}} \times \mathcal{O}_{l,1}$, whereas the remaining edges in $\mathcal{O}_{k,p^{k}} \times \mathcal{O}_{k,1}$ are mapped to edges in $\mathcal{O}_{l,p^{l}} \times \mathcal{O}_{l,p}$. Thus if $x \neq y$ are two elements of $\mathcal{O}_{k,1}$ such that $x$ is mapped to $x' \in \mathcal{O}_{l,1}$ and $y$ is mapped to $y' \in \mathcal{O}_{l,p}$, then the following equation may have no solution,

\begin{center}
$x + y(mod~p^k) = x' + y'(mod~p^l)$.
\end{center}
\end{rem}

\begin{defn} \label{df2} Let $\Gamma_1$ and $\Gamma_2$ be two simple graphs. Then $(a, b) \in E(\Gamma_1)$ degenerates to  $(u, v) \in E(\Gamma_2)$ if there exists a homomorphism $\varphi : V(\Gamma_1) \longrightarrow V(\Gamma_2)$ such that $\varphi(a, b) = (u, v)$. If every edge of $\Gamma_1$ degenerates to edges in $\Gamma_2$, then we say that $\Gamma_1$ degenerates to $\Gamma_2$.
\end{defn}

Recall that an \textit{independent part} (independent set) in a graph $\Gamma$ is a set of vertices of $\Gamma$ such that for every two vertices, there is no edge in $\Gamma$ connecting the two. Also, the \textit{complete part} (complete subgraph) in a graph $\Gamma$ is a set of vertices in $\Gamma$ such that there is an edge between every pair of vertices in $\Gamma$.

The simplified form of Lemma \eqref{lm1} is presented in the following result. We adapted the definition of degeneration in groups and make it to work for graphs which are realized by finite abelian groups.

\begin{thm}\label{thm1}
If under any graph homomorphism $\mathcal{O}_{k,p^{k}}$ is the only vertex mapped to $\mathcal{O}_{l,p^{l}}$, then the pair $(p^ru, p^su) \in \mathcal{O}_{k,p^{r}} \times \mathcal{O}_{k,p^s}$ degenerates to $(p^{r'}u, p^{s'}u) \in \mathcal{O}_{l,p^{r'}} \times \mathcal{O}_{l,p^{s'}}$ if and only if $r\leq r'$ and $s\leq s'$, where $u$ is relatively prime to $p$ and $k\leq l$.
\end{thm}

\begin{proof}
In setting of the symmetric group action on finite abelain $p$-groups of rank one, let $\mathcal{O}_{k,p^{r}}$,  $\mathcal{O}_{k,p^s}$ be orbits represented by elements $p^r$ and $p^s$ of the source group and $\mathcal{O}_{l,p^{r'}}$,  $\mathcal{O}_{l,p^{s'}}$ be orbits represented by elements $p^{r'}$ and $p^{s'}$ of the target group, where $0\leq r, ~s\leq k-1$ and $0\leq r',~ s'\leq l-1$. We consider the cases hereunder.

\textbf{Case I:} $k = l = 2t$, $t\in\mathbb{Z}_{>0}$. Then the independent and complete parts of the graph realised by a source group is $X = \dot\bigcup_{i=0}^{t-1}\mathcal{O}_{k, p^i}$ and $Y = \dot\bigcup_{i=0}^{t-1}\mathcal{O}_{k, p^{t+j}}$, where each element of both $X$ and $Y$ are connected to $\mathcal{O}_{k,p^{k}} =\{0\}$. Similarly, $X' = \dot\bigcup_{i=0}^{t-1}\mathcal{O}_{l, p^i}$ and $Y' = \dot\bigcup_{j=0}^{t-1}\mathcal{O}_{l, p^{t+j}}$ represents the independent and complete parts of the graph realized by a target group, where each element of both $X'$ and $Y'$ are connected to $\mathcal{O}_{l,p^{l}} =\{0\}$.

Let $x\in X$. If $x\in \mathcal{O}_{k, 1}$, then as discussed above, $x$ is adjacent to $\mathcal{O}_{k, p^k}$ only. On the other hand, if $x\in \mathcal{O}_{k, p^i}$ for $1\leq i \leq t-1$, then $x$ is adjacent to all elements of the set $\dot\bigcup_{n=i}^{0}\mathcal{O}_{k, p^{k-n}} \subset Y$. Moreover, if $x'\in \mathcal{O}_{l, 1}$, then $x'$ is adjacent to $\mathcal{O}_{l, p^l}$ whereas if $x'\in \mathcal{O}_{l, p^j}$ for $1\leq j \leq t-1$, then $x'$ is adjacent to all elements of the set $\dot\bigcup_{m=j}^{0}\mathcal{O}_{l, p^{l-m}} \subset Y'$. Under any given graph homomorphism $\tau$, the images of relations in $X \times \mathcal{O}_{k, p^k}$, $X \times Y$ and $Y \times \mathcal{O}_{k, p^k}$ are in $X' \times \mathcal{O}_{l, p^l}$, $X' \times Y'$ and $Y' \times \mathcal{O}_{l, p^l}$. Let $(a, b) \in X \times \mathcal{O}_{k, p^k} \bigcup X \times Y \bigcup Y \times \mathcal{O}_{k, p^k}$. Suppose $(a, b)$ degenerates to some $(a', b') \in X' \times \mathcal{O}_{l, p^l} \bigcup X' \times Y' \bigcup Y' \times \mathcal{O}_{l, p^l}$ . If $\tau$ is group homomorphism such that $\tau(1) \in \mathcal{O}_{l, 1}$, then 
\begin{center}
$\mathcal{O}_{k,1} \times \mathcal{O}_{k,p^{k}} \xhookrightarrow{1-1} \mathcal{O}_{l,1}\times\mathcal{O}_{l,p^{l}}$, 

$\mathcal{O}_{k,p} \times \mathcal{O}_{k,p^{k}} \xhookrightarrow{1-1} \mathcal{O}_{l,p}\times\mathcal{O}_{l,p^{l}}$,

$\mathcal{O}_{k,p} \times \mathcal{O}_{k,p^{k-1}} \xhookrightarrow{1-1} \mathcal{O}_{l,p}\times\mathcal{O}_{l,p^{l-1}}$, 

\hskip .1cm \vdots
\end{center}

If $\tau(1)\in \mathcal{O}_{l, p}$, then

\begin{center}
$\mathcal{O}_{k,1} \times \mathcal{O}_{k,p^{k}} \twoheadrightarrow  \mathcal{O}_{l,p} \times \mathcal{O}_{l,p^{l}} $,

$\mathcal{O}_{k,p} \times \mathcal{O}_{k,p^{k}} \twoheadrightarrow  \mathcal{O}_{l,p^{2}} \times \mathcal{O}_{l,p^{l}} $,

$\mathcal{O}_{k,p} \times \mathcal{O}_{k,p^{k-1}} \twoheadrightarrow  \mathcal{O}_{l,p^{2}} \times \mathcal{O}_{l,p^{l-1}}$,

\hskip .1cm \vdots
\end{center}

If $\tau(1)$ lies in any other orbit of $X'\bigcup Y'$, then as above we have the mapping of edges to edges. Thus for any group homomorphism which maps $(p^ru, p^su) \in \mathcal{O}_{k,p^{r}} \times \mathcal{O}_{k,p^s}$ to $(p^{r'}u, p^{s'}u) \in \mathcal{O}_{l,p^{r'}} \times \mathcal{O}_{l,p^{s'}}$, the relations $r\leq r'$ and $s\leq s'$ are verified.

Now, suppose $\tau$ is not a group homomorphism but a graph homomorphism. Assume without loss of generality that under $\tau$, $A \times \mathcal{O}_{k,p^{k}} \xhookrightarrow{1-1} A' \times \mathcal{O}_{l,p^{l}}$, where $A \subset \mathcal{O}_{k,1} \subset X$ and $A' \subset \mathcal{O}_{l,1} \subset X'$ are proper subsets of $X$ and $X'$. Moreover,
\begin{center}
$\mathcal{O}_{k,1}\setminus A \times \mathcal{O}_{k,p^{k}} \bigcup \mathcal{O}_{k,p} \times \mathcal{O}_{k,p^{k}} \bigcup \mathcal{O}_{k,p} \times \mathcal{O}_{k,p^{k-1}} \twoheadrightarrow \mathcal{O}_{l,p} \times \mathcal{O}_{l,p^{l-1}} \bigcup \mathcal{O}_{l,p} \times \mathcal{O}_{l,p^{l}}$,

$\mathcal{O}_{k,p^2} \times \mathcal{O}_{k,p^{k}} \xhookrightarrow{1-1} \mathcal{O}_{l,p^2} \times \mathcal{O}_{l,p^{l}}$,

$\mathcal{O}_{k,p^2} \times \mathcal{O}_{k,p^{k-1}} \xhookrightarrow{1-1} \mathcal{O}_{l,p^2} \times \mathcal{O}_{l,p^{l-1}}$,

$\mathcal{O}_{k,p^2} \times \mathcal{O}_{k,p^{k-2}} \xhookrightarrow{1-1} \mathcal{O}_{l,p^2} \times \mathcal{O}_{l,p^{l-2}}$,

\hskip .1cm \vdots

\end{center}
Thus, for $\tau$, we observe that the relations $r\leq r'$ and $s\leq s'$ hold. Similarly these relations can be verified for other graph homomorphisms.

Suppose to the contrary that $r > r'$ and $s > s'$. Then $(a, b)$ does not degenerates to  $(a', b')$, since by Lemma \eqref{lm1}, $a$ and $b$ degenerates to $a'$ and $b'$ if and only if $r\leq r'$ and $s\leq s'$, therefore, a contradiction. Further, if under any graph homomorphism the elements of orbits $\mathcal{O}_{k,p^r} \times \mathcal{O}_{k,p^{s}}$ are mapped to elements of $\mathcal{O}_{l,p^{r'}} \times \mathcal{O}_{l,p^{s'}}$, then it follows that for some $1 \leq s \leq k-1$, $\mathcal{O}_{k,p^{s}}$ is mapped to $\mathcal{O}_{l,p^{l}}$, again a contradiction.

\textbf{Case II:} $k = l = 2t+1$, $t\in\mathbb{Z}_{>0}$. The independent and complete parts of the graph realised by  source and target groups are $X = \dot\bigcup_{i=0}^{t}\mathcal{O}_{k, p^i}$, $Y = \dot\bigcup_{j=1}^{t+1}\mathcal{O}_{l, p^{t+j}}$ and $X' = \dot\bigcup_{i=0}^{t}\mathcal{O}_{l, p^i}$, $Y' = \dot\bigcup_{j=0}^{t+1}\mathcal{O}_{l, p^{t+j}}$. Rest of the proof for this case follows by the same argument which we discussed above for the even case.

Finally, if we consider the cases $(k, l) = (2t, 2t+1)$ or $(k, l) = (2t+1, 2t)$, then these cases can be handled in the same manner as above.
\end{proof}

\begin{figure}[H]
\begin{center}
 \includegraphics[scale=.360]{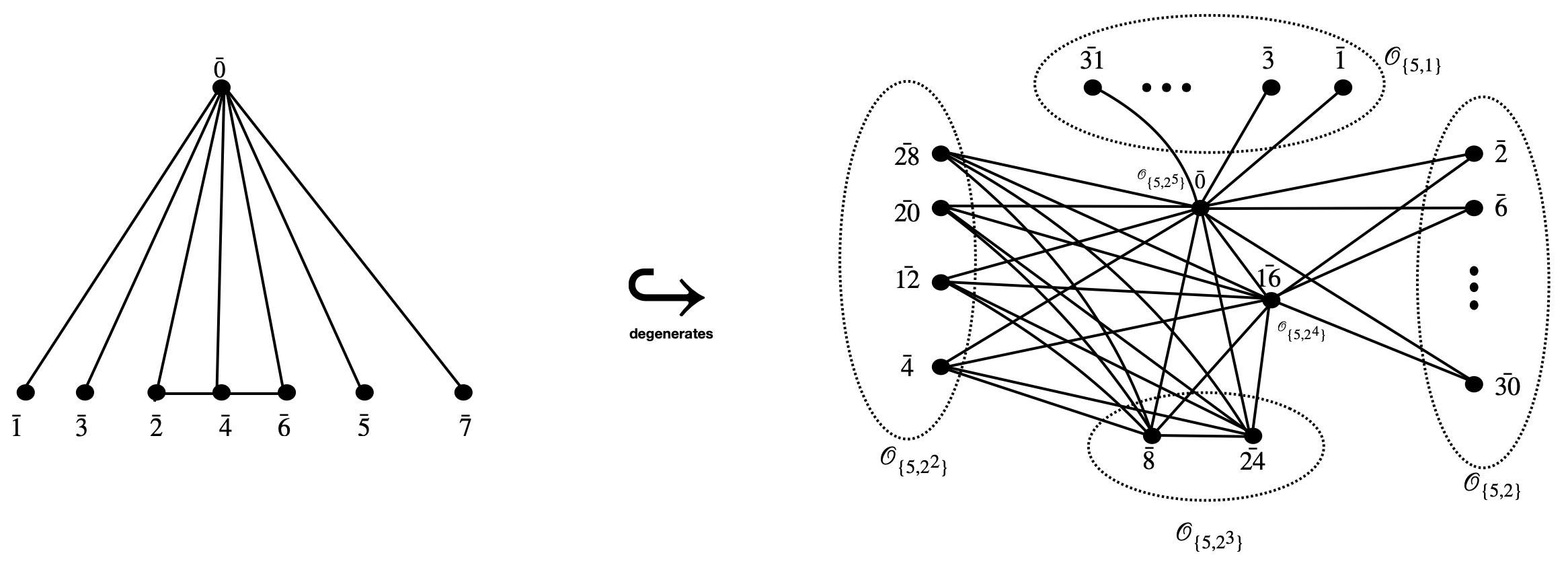}
\end{center}
\caption{Pictorial sketch of degeneration}
\label{2}
\end{figure}

Note that in Figure \eqref{2}, the graph on the left hand side is the graph realized by $\Z/2^3\Z$ and the graph on the right hand side is realized by $\Z/2^5\Z$.
 
\section{\bf{Partial orders on $\mathcal{T}_{p_1 \cdots p_n}$}}
In this section, we study some relations on the set $\mathcal{T}_{p_1 \cdots p_n}$ of all graphs realized by finite abelian $p_r$-groups of rank $1$, where each $p_r$, $1\leq r\leq n$, is a prime number. We discuss equivalent forms of the partial order "degeneration" on $\mathcal{T}_{p_1 \cdots p_n}$ and obtain a locally finite distributive lattice of finite abelian groups.

Threshold graphs play an essential role in graph theory as well as in several applied areas which include psychology and computer science \cite{MP}. These graphs were introduced by Chv\'{a}tal and Hammer \cite{CH}  and Henderson and Zalcstein \cite{HZ}.

A vertex in a graph $\Gamma$ is called \textit{dominating} if it is adjacent to every other vertex of $\Gamma$. A graph $\Gamma$ is called a \textit{threshold graph} if it is obtained by the following procedure.

Start with $K_1$, a single vertex, and use any of the following steps, in any order, an arbitrary number of times.

(i) Add an isolated vertex.

(ii) Add a dominating vertex, that is, add a new vertex and make it adjacent to each existing vertex.

It is always interesting to determine the classes of threshold graphs, since we may represent a threshold graph on $n$ vertices using a binary code $(b_1, b_2, \cdots, b_n)$, where $b_i = 0$ if vertex $v_i$ is being added as an isolated vertex and $b_i =  1$ if $v_i$ is being added as a dominating vertex. Furthermore, using the concept of creation sequences we establish the nullity, multiplicity of some non-zero eigenvalues and the Laplacian eigenvalues of a threshold graph. The Laplacian eigenvalues of $\Gamma$ are the eigenvalues of a matrix $D(\Gamma) - A(\Gamma)$, where  $D(\Gamma)$ is the diagonal matrix of vertex degrees and $A(\Gamma)$ is the familiar $(0, 1)$ adjacency matrix of $\Gamma$.

The authors in \cite{ER} confirmed that the graph realised by a finite abelian $p$-group of rank $1$ is a threshold graph. In fact, they proved the following intriguing result for a finite abelain $p$-groups of rank $1$.

\begin{thm}\cite{ER}.\label{thm2}
If $G$ is a finite abelian $p$-group of rank $1$, then $\Gamma(G)$ is a threshold graph.
\end{thm}
Let $p_1 < p_2 < \cdots < p_n$ be a sequence of primes and let $\lambda_{i} = (\lambda_{i,1}, \lambda_{i, 2}, \cdots, \lambda_{i, n})$ be sequence of partitions of positive integers, where $1 \leq i \leq n$. For each prime $p_t$, where $1 \leq t \leq n$,  the sequences of finite abelian $p_t$-groups with respect to partitions $\lambda_{i,1}, \lambda_{i, 2}, \cdots, \lambda_{i, n}$ are listed as follows,
\begin{center}
$G_{\lambda_1, p_1} = \mathbb{Z}/p_{1}^{\lambda_{1, 1}}\mathbb{Z} \oplus \mathbb{Z}/p_{1}^{\lambda_{1, 2}}\mathbb{Z} \oplus \cdots \oplus \mathbb{Z}/p_{1}^{\lambda_{1,n}}\mathbb{Z}$,

$G_{\lambda_2, p_2} = \mathbb{Z}/p_{2}^{\lambda_{2, 1}}\mathbb{Z} \oplus \mathbb{Z}/p_{2}^{\lambda_{2, 2}}\mathbb{Z} \oplus \cdots \oplus \mathbb{Z}/p_{2}^{\lambda_{2,n}}\mathbb{Z}$,

\hskip .1cm \vdots

\end{center}

Fix a prime $p_r$, where $1\leq r \leq n$. Then for each distinct power $\lambda_{i, j}$, $1 \leq i, j \leq n$, it follows from Theorem \eqref{thm2}, that members of the sequence of graphs realised by a sequence of finite abelian $p_r$-groups of rank $1$ are threshold graphs. The sets of orbits of symmetric group action on sequence of finite abelian $p_r$-groups $\mathbb{Z}/p_{r}^{\lambda_{r, 1}}\mathbb{Z}, \mathbb{Z}/p_{r}^{\lambda_{r, 2}}\mathbb{Z}, \cdots, \mathbb{Z}/p_{r}^{\lambda_{r, n}}\mathbb{Z}$ of rank $1$ are: 
\begin{center}
$\{\mathcal{O}_{r, 1}, \mathcal{O}_{r, p_r^{1}}\}$,

 $\{\mathcal{O}_{r, 1}, \mathcal{O}_{r, p_r^{1}}, \mathcal{O}_{r, p_r^{2}}\}$,
 
$\{\mathcal{O}_{r, 1}, \mathcal{O}_{r, p_{r}^{1}}, \mathcal{O}_{r, p_r^{2}}, \mathcal{O}_{r, p_r^{3}}\}$,

\hskip .1cm \vdots
\end{center}
Note that, $\lambda_{r, 1} = 1, \lambda_{r, 2} = 2, \lambda_{r, 3} = 3, \cdots$, in the above sequence of finite abelian$p_r$-groups. 

Thus for each prime $p_r$ and positive integer $\lambda_{i, j}$, we have sequences of threshold graphs realised by sequences of abelian $p_r$-groups. 

The \textit{degree sequence} of a graph $\Gamma$ is given by $\pi(\Gamma) = (d_1, d_2, \cdots, d_n)$, which is the non-increasing sequence of non-zero degrees of vertices of $\Gamma$. 

For a graph $\Gamma$ of order $n$ and size $m$, let $d = [d_1, d_2, \cdots, d_n]$ be a sequence of non-negative integers arranged in non-increasing order, which we refer to as a partition of $2m$. Define the transpose of the partition as $d^* = [d_1^{*}, d_2^{*}, \cdots, d_r^{*}]$, where $d_j^{*} = |\{d_i : d_i \geq j\}|$, $j = 1, 2,\cdots, r$. Therefore $d_j^{*}$ is the number of $d_i$'s that are greater than equal to $j$. Recall from \cite{RB} that a sequence $d^*$ is called the conjugate sequence of $d$. The another interpretation of a conjugate sequence is the \textit{Ferrer's diagram (or Young diagram)} denoted by $Y(d)$ corresponding to $d_1, d_2, \cdots, d_n$ consists of $n$ left justified rows of boxes, where the $i^{th}$ row consists of $d_i$ boxes (blocks), $i = 1, 2, \cdots, n$. Note that $d_i^{*}$ is the number of boxes in the $i^{th}$ column of the Young diagram with $i = 1, 2, \cdots, r$. An immediate consequence of this observation is that if $d^*$ is the conjugate sequence of $d$, then,
\begin{equation*}
\sum\limits_{i = 1}^{n} d_i = \sum\limits_{i = 1}^{r} d_i^{*}
\end{equation*}

If $d$ represents the degree sequence of a graph, then the number of boxes in the $i^{th}$ row of the Young diagram is the degree of vertex $i$, while the number of boxes in the $i^{th}$ row of the Young diagram of the transpose is the number of vertices with degree at least $i$. The trace of a Young diagram $tr(Y(d))$ is $tr(Y(d)) = |\{i : d_i\geq i\}| = tr(Y(d^{*}))$, which is the length of "diagonal" of the Young diagram for $d$ (or $d^{*}$).

The degree sequence is a graph invariant, so two isomorphic graphs have the same degree sequence. In general, the degree sequence does not uniquely determine a graph, that is, two non-isomorphic
graphs can have the same degree sequence. However, for threshold graphs, we have the following result.

\begin{prop}[\cite{RMB}]\label{prp2} Let $\Gamma_1$ and $\Gamma_2$ be two threshold graphs and let $\pi_1(\Gamma_1)$ and $\pi_{2}(\Gamma_2)$ be  degree sequences of $\Gamma_1$ and $\Gamma_2$ respectively. If $\pi_1(\Gamma_1) = \pi_{2}(\Gamma_2)$, then $\Gamma_1 \cong \Gamma_2$.
\end{prop}
The Laplacian spectrum of threshold graphs $\Gamma$, which we denote by $\ell-spec(\Gamma)$, have been studied in \cite{HK, RM}. In \cite{HK}, the formulas for the Laplacian spectrum, the Laplacian polynomial, and the number of spanning trees of a threshold graph are given. It is shown that the degree sequence of a threshold graph and the sequence of eigenvalues of its Laplacian matrix are \enquote {almost the same} and on this basis, formulas are given to express the Laplacian polynomial and the number of spanning trees of a threshold graph in terms of its degree sequence. 

The following is the fascinating result regarding the Laplacian eigenvalues of the graph realized by a finite abelian $p$-group of rank $1$.

\begin{thm}\cite{ER}. \label{thm3}
Let $\Gamma(G)$ be the graph realized by a finite abelian $p$-group of the type $G = \mathbb{Z}/p^{k}\mathbb{Z}$. Then the representatives $0, 1, p, p^2, \cdots, p^{k - 1}$ (with multiplicities) of orbits $\{\mathcal{O}_{k, p^{k}}\} \cup \{\mathcal{O}_{k, p^{i}} : 0\leq i \leq k - 1\}$ of symmetric action on $G$ are the Laplacian eigenvalues of $\Gamma(G)$, that is, $\ell-spec(\Gamma(G)) = \{0, 1, p, p^2, \cdots, p^{k - 1}, p^k \}$.
\end{thm}

\begin{defn}\label{df3}
Let $\pi_1, \pi_2, \cdots, \pi_n \in \mathbb{Z}_{>0}$ and $\pi_1^{\bullet}, \pi_2^{\bullet}, \cdots, \pi_n^{\bullet} \in \mathbb{Z}_{>0}$ be some partitions of $n \in \mathbb{Z}_{>0}$. A sequence (partition) of eigenvalues $\pi = (\pi_1, \pi_2, \cdots, \pi_n)$ of a graph $\Gamma$ is said to be a threshold eigenvalues sequence (partition) if $\pi_{i}  = \pi_{i}^{\bullet} + 1$ for all $i$ with $1 \leq i \leq tr(Y(\pi))$. 
\end{defn}

Just for the convenience we refer the Laplacian eigenvalues as eigenvalues. The sequence of representatives of orbits (or eigenvalues of $\Gamma(\mathbb{Z}/p^{k}\mathbb{Z})$) of a symmetric action on a group $\mathbb{Z}/p^{k}\mathbb{Z}$ obtained in Theorem \eqref{thm3} represents transpose of a young diagram $Y(d)$, where $d$ is the degree sequence of the graph realized by $\mathbb{Z}/p^{k}\mathbb{Z}$.

For a group $G = \mathbb{Z}/2^4\mathbb{Z}$ be a group, the degree sequence $\sigma$ of $\Gamma(G)$ is,
\begin{equation*}
\sigma = \pi^{\bullet} = (15, 7, 3, 3, 2, 2, 2, 2, 1, 1, 1, 1, 1, 1, 1, 1).
\end{equation*}

The conjugate sequence of $\sigma$ is,
\begin{equation*} 
\sigma^{*} = \pi = (2^4, 2^3, 2^2, 2, 2, 2, 2, 1, 1, 1, 1, 1, 1, 1, 1).
\end{equation*}

A partition $\pi$ of eigenvalues of $\Gamma(G)$ is a threshold eigenvalues partition, since $\sum\limits_{i = 1}^{3}\pi_{i} =  \sum\limits_{i = 1}^{3}\pi_{i}^{\bullet} + 1$. Note that $tr(Y(\pi)) = 3$, the three blocks in $Y(\sigma^{*}) = Y(\pi)$ are shown as $t_{11}, t_{22}, t_{33}$ before the darkened column in Figure \eqref{3} below.

\begin{figure}[H]
 \begin{center}
 \includegraphics[scale=.400]{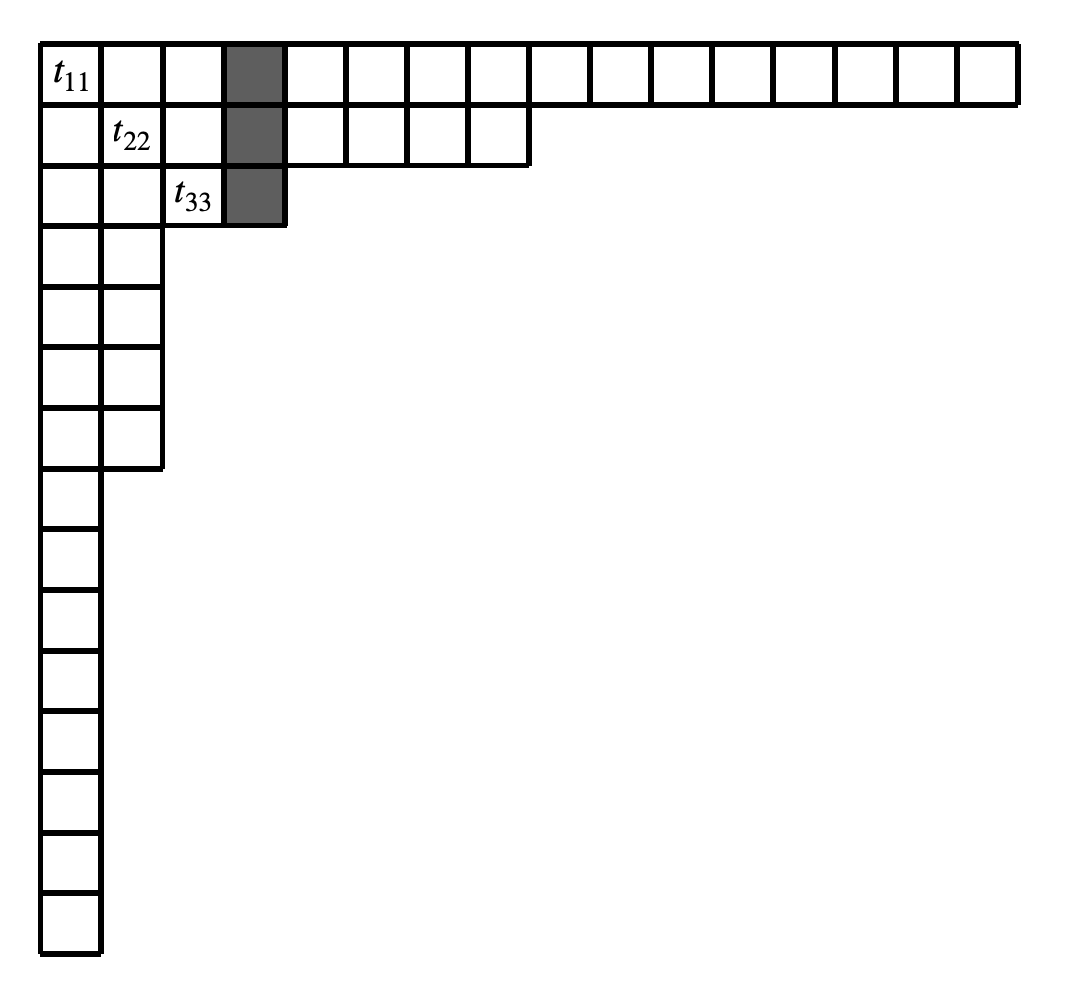}
 \end{center}
\caption{\bf{$Y(\pi)$}}
\label{3}
\end{figure}

Thus from above discussion we assert that a partition $\pi$ of eigenvalues is a threshold eigenvalues partition if and only if $Y(\pi)$ can be decomposed into an $tr(Y(\pi)) \times tr(Y(\pi))$ array of blocks in the upper left-hand corner called the \textit{trace square} in $Y(\pi)$. A column of $tr(Y(\pi))$ blocks placed immediately on the right hand side of trace square, darkened in Figure \eqref{3}, and a piece of blocks on the right hand side of column $tr(Y(\pi)) + 1$ is the transpose of the piece which is below the trace square. 
 
If $a = (a_1, a_2, \cdots, a_r)$ and $b = (b_1, b_2, \cdots, b_s)$ are non-increasing sequences of real numbers. Then $b$ \textit{weakly majorizes} $a$, written as $b \succeq a$, if $r \geq s$,

\begin{equation} 
\sum\limits_{i = 1}^{k}b_i \geq \sum\limits_{i = 1}^{k}a_i, 
\label{eq1}   
\end{equation}
where $1 \leq k \leq s$, and 
\begin{equation} 
\sum\limits_{i = 1}^{r}b_i \geq \sum\limits_{i = 1}^{s}a_i. 
\label{eq2}      
\end{equation}

If $b$ weakly majorizes $a$ and equality holds in \eqref{eq2}, then $b$ \textit{majorizes} $a$, written as $b \succ a$.

We present an example which illustrates that the threshold eigenvalues partition of some graph realized by a finite abelian $p$-group $G_1$ majorizes the degree partition of the graph realized by some other finite abelian $p$-group $G_2$.

Let $G_1 = \mathbb{Z}/2^3\mathbb{Z}$ and $G_1 = \mathbb{Z}/3^2\mathbb{Z}$ be two groups. The degree partitions $\pi_1^{\bullet}$ and $\pi_2$ of graphs $\Gamma(G_1)$ and $\Gamma(G_2)$ are listed below as,
\begin{center}
$\pi_1^{\bullet} = (7, 3, 2, 2, 1, 1, 1, 1),$

$\pi_2 = (8, 2, 2, 1, 1, 1, 1, 1, 1).$
\end{center} 

The partitions $\pi_1^{\bullet}, \pi_2 \in \mathcal{P}(18)$, where $\mathcal{P}(18)$ is the set of all partitions of $18$. The partition $\pi_1 = (8, 4, 2, 1, 1, 1, 1)$ is the threshold eigenvalues partition of $\Gamma(G_1)$. The Young diagrams of partitions $\pi_1$ and $\pi_{2}$ are shown in Figure \eqref{4}. 

\begin{figure}[H]
 \begin{center}
 \includegraphics[scale=.310]{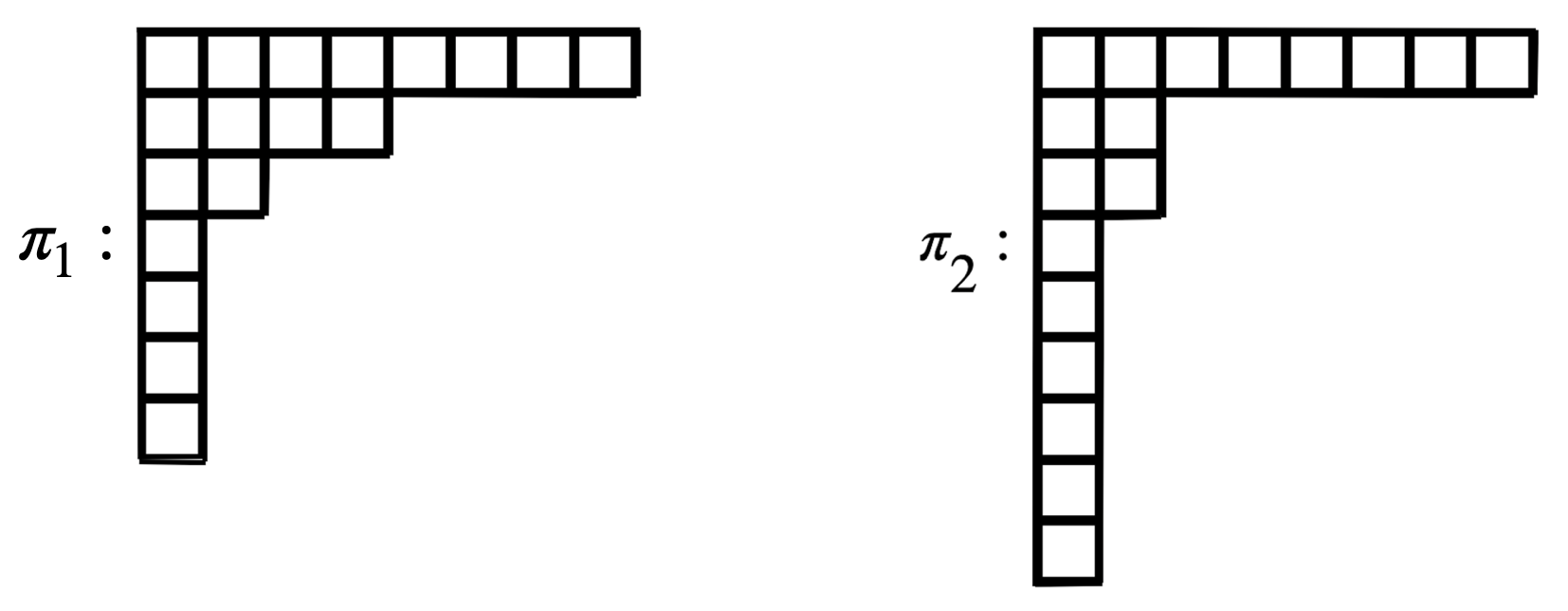}
 \end{center}
\caption{Young diagrams of $\pi_1$ and $\pi_{2}$}
\label{4}
\end{figure}

Let $\pi^{\bullet}$ and $\sigma$ be two degree sequences of graphs realized by finite abelian $p$-groups of rank $1$ such that $\pi^{\bullet}, \sigma \vdash m$, where $m\in \mathbb{Z}_{> 0}$. Then $\pi \succ \sigma$ if and only if $Y(\pi)$ can be obtained from $Y(\sigma)$ by moving blocks of the highest row in $Y(\sigma)$ to lower numbered rows. Thus majorization induces a partial order on sets $\{Y(\pi^{\bullet}) : \pi^{\bullet} ~is ~a ~degree ~sequence ~of ~some ~graph ~re$
\noindent $alized ~by ~a ~p-group ~of ~rank ~1\}$ and $\{Y(\pi^{\bullet}) : \pi^{\bullet} \vdash n, n \in \mathbb{Z}_{>0}\}$.

\begin{cor}
If $\pi, \sigma \in \mathcal{P}(n)$, $n\in \mathbb{Z}_{>0}$, then $\pi \succ \sigma$ if and only if $Y(\pi)$ can be obtained from $Y(\sigma)$ by moving blocks of the highest row in $Y(\sigma)$ to lower numbered rows.
\end{cor}

\begin{thm}\label{thm4} Let $\mathcal{T}_{p_1 \cdots p_n}$ be the collection of all graphs realised by all sequences of finite abelian $p_r$-groups, where $1 \leq r \leq n$. If $\pi$ is a threshold eigenvalues partition, then upto isomorphism, there is exactly one finite abelian $p_r$-group $G$ of rank $1$ such that $\ell-spec(\Gamma(G))\setminus \{0\} = \pi$.
\end{thm}
\begin{proof}
Let $\left(\Gamma(\mathbb{Z}/p_{r}^{\lambda_{r, 1}}\mathbb{Z}), \Gamma(\mathbb{Z}/p_{r}^{\lambda_{r, 2}}\mathbb{Z}), \cdots, \Gamma(\mathbb{Z}/p_{r}^{\lambda_{r, n}}\mathbb{Z})\right) \in \mathcal{T}_{p_1 \cdots p_n}$ be a sequence of graphs realized by a sequence of finite abelian $p_r$-groups $\left(\mathbb{Z}/p_{r}^{\lambda_{r, 1}}\mathbb{Z}, \mathbb{Z}/p_{r}^{\lambda_{r, 2}}\mathbb{Z}, \cdots, \mathbb{Z}/p_{r}^{\lambda_{r, n}}\mathbb{Z}\right)$. 

Let $\pi$ be a threshold eigenvalues partition of some graph of the sequence. Without loss of generality let it be the graph realised by a finite abelian $p_r$-group $\mathbb{Z}/p_{r}^{\lambda_{r, r}}\mathbb{Z}$. The partition $\pi$ is represented by Young diagram $Y(\pi)$ and the Young diagram for the abelian $p_r$-group of type $\mathbb{Z}/p_{r}^{\lambda_{r, r-1}}\mathbb{Z}$ can be obtained from $Y(\pi)$ by removing some blocks in rows and columns of $Y(\pi)$. The proof now follows by induction on terms of the sequence of graphs.
\end{proof}

For $1\leq i \leq j \leq n$, let $G$ be a finite abelian $p_i$-group of rank $1$ and $H$ be a finite abelian $p_j$-group of the same rank. Moreover, let $\Gamma(G)$ and $\Gamma(H)$ be two graphs realized by $G$ and $H$. We define a partial order "$\leq$" on $\mathcal{T}_{p_1 \cdots p_n}$. Graphs $\Gamma(G), \Gamma(H) \in \mathcal{T}_{p_1 \cdots p_n}$ are related as $\Gamma(G) \leq  \Gamma (H)$ if and only if $\Gamma(H)$ contains a subraph isomorphic to $\Gamma(G)$, that is if and only if $\Gamma(G)$ can be obtained from $\Gamma(H)$  by "deletion of vertices".

The relation "degeneration" on the set $\mathcal{T}_{p_1 \cdots p_n}$ descends to a partial order on $\mathcal{T}_{p_1 \cdots p_n}$ and two graphs $\Gamma(G)$, $\Gamma(H)$ are related if $\Gamma(G)$ degenerates to $\Gamma(H)$. It is not hard to verify that the partial orders "$\leq$" and "degeneration" are equivalent on $\mathcal{T}_{p_1 \cdots p_n}$, since by "deletion of vertices" in $\Gamma(H)$ we get the homomorphic image of $\Gamma(G)$ in $\Gamma(H)$ and if $\Gamma(G)$ degenerates to $\Gamma(H)$, then $\Gamma(G)$ can be obtained from $\Gamma(H)$ by "deletion of vertices".

Recall that a poset $P$ is locally finite if the interval $[x, z] = \{y \in P : x \leq y \leq z\}$ is finite for all $x, z \in P$. If $x, z \in P$ and $[x, z] = \{x, z\}$, then $z$ \textit{covers} $x$. A Hasse diagram of $P$ is a graph whose vertices are the elements of $P$, whose edges are the cover relations, and such that z is drawn "above" x whenever $x < z$.

A lattice is a poset $P$ in which every pair of elements $x, y \in P$ has a least
upper bound (or join), $x \vee y \in P$, and a greatest lower bound (or meet), $x \wedge y \in P$. Lattice $P$ is distributive if $x \wedge (y \vee z) = (x \wedge y) \vee (x \wedge z)$ and
$x \vee (y \wedge z) = (x \vee y) \wedge (x \vee z)$ for all $x, y, z \in P$.

Let $\mathcal{Y}$ be the set of all threshold eigenvalues partitions of members of $\mathcal{T}_{p_1 \cdots p_n}$. If $\mu, \eta \in \mathcal{Y}$, define $\mu \leq \eta$, if $Y(\mu)$ "fits in" $Y(\eta)$, that is, if $\mu \leq \eta$, then $Y(\eta)$ is overlapped by $Y(\mu)$ or $Y(\mu)$ fits inside $Y(\eta)$.  The set $\mathcal{Y}$ with respect this partial ordering is a locally finite distributive lattice. The unique smallest element of $\mathcal{Y}$ is $\hat{0} = \emptyset$, the empty set.

Recall that the dual of a poset $P$ is the poset $P^{*}$ on the same set as $P$, such that $x \leq y$ in $P^{*}$ if and only if $y \leq x$ in $P$. If $P$ is isomorphic to $P^{*}$, then $P$ is self-dual.

\begin{thm}\label{thm5}
If $\Gamma(G), \Gamma(H) \in \mathcal{T}_{p_1 \cdots p_n}$, then $\Gamma(G)\leq \Gamma(H)$ if and only $Y(\mu)$ "fits in" $Y(\eta)$, where $\mu$ and $\eta$ are threshold eigenvalues partitions of graphs $\Gamma(G)$ and $\Gamma(H)$.
\end{thm} 

\begin{proof}
If $\Gamma(G)$ is obtained from $\Gamma(H)$ by deletion of one or more vertices, then the terms in the threshold eigenvalues partition $\mu$ are less in number than the terms in the threshold eigenvalues partition $\eta$ of $\Gamma(H)$. It follows that $Y(\mu)$ "fits in" $Y(\eta)$.

Conversely, suppose $Y(\mu)$ "fits in" $Y(\eta)$. The threshold eigenvalues partitions $\mu$ and $\eta$ are obtained from degree sequences of $\Gamma(G)$ and $\Gamma(H)$. If $\Gamma(G)$ and $\Gamma(H)$ have same degree sequence, then $\mu = \eta$. Therefore by Proposition \eqref{prp2}, $\Gamma(G) \cong \Gamma(H)$. Otherwise, $\mu \neq \eta$. Let $\Gamma(K)$ be a subgraph of $\Gamma(H)$ obtained by removing a pendant vertex from $\Gamma(H)$. Then $Y(\eta')$ is obtained from $Y(\eta)$ by removing a single block in the string with number of blocks in the string equal to the largest eigenvalue in $\eta$. It is clear that $Y(\eta'$ "fits in" $Y(\eta)$. We continue the process of deletion of vertices untill the resulting graph has the same threshold eigenvalues partition as $\Gamma(G)$. Thus, it follows that $\Gamma(H)$ contains a subgraph isomorphic to $\Gamma(H)$, that is, $\Gamma(G) \leq \Gamma(H)$.
\end{proof}

\begin{cor}\label{cr1}
The sets $\mathcal{T}_{p_1 \cdots p_n}$ and $\mathcal{Y}$ are isomorphic to each other (as posets).
\end{cor}
\begin{proof}
The bijection $\Gamma(G) \longrightarrow Y(\mu)$ is a poset isomorphism from $\mathcal{T}_{p_1 \cdots p_n}$ onto $\mathcal{Y}$, where $\mu$ is threshold eigenvalues partition of the graph $\Gamma(G) \in \mathcal{T}_{p_1 \cdots p_n} $ realised by a finite abelian$p_r$-group of rank $1$. 
\end{proof}

For $n \geq 1$, let $\mathcal{F}_n$ be the the collection of all connected threshold graphs on $n$ vertices. We extend the partial order "$\leq$" to $\mathcal{F}_n$. Two graphs $G_1, G_2 \in\mathcal{F}_n$ are related as $G_1 \leq G_2$ if and only if $G_1$ is isomorphic to a subgraph of $G_2$. It is not difficult to verify that the poset $\mathcal{T}_{p_1 \cdots p_n}$ is an induced subposet of $\mathcal{F}_n$ and $\mathcal{F}_n$ is a self-dual distributive lattice. Moreover, if $\mathcal{H}_n$ is the collection of threshold eigenvalues partitions of members of $\mathcal{F}_n$, then again it is easy verify that $\mathcal{H}_n$ is a poset with respect to partial order "fits in" and  we have the following observation related to posets $\mathcal{F}_n$ and $\mathcal{H}_n$.

\begin{cor}
The bijection $G \longrightarrow Y(\mu)$ is a poset isomorphism from $\mathcal{F}_n$ to $\mathcal{H}_n$, where $\mu$ is threshold eigenvalues partition of $G \in \mathcal{F}_n$. In particular, $\mathcal{H}_n$ is self-dual distributive lattice.
\end{cor}

Now, we focus on sub-sequences (sub-partitions) of a threshold eigenvalues partition. We begin by dividing $Y(\pi)$ into two disjoint pieces of blocks, where $\pi)$ is a threshold eigenvalues partition of a graph $\Gamma(G)\in \mathcal{T}_{p_1 \cdots p_n}$. We denote by $R(Y(\pi))$ those blocks of $Y(\pi)$ which lie on the diagonal of a trace square of $Y(\pi)$ and to the right of diagonals. By the notation $C(Y(\pi))$, we denote those blocks of $Y(\pi)$ that lie strictly below diagonals of a trace square, that is, $R(Y(\pi))$ is a piece of blocks of $Y(\pi)$ on or above the diagonal and $C(Y(\pi))$ is the piece of $Y(\pi)$ which lie strictly below the diagonal. This process if division is illustrated as follows (Figure \eqref{5}).

\begin{figure}[H]
 \begin{center}
 \includegraphics[scale=.350]{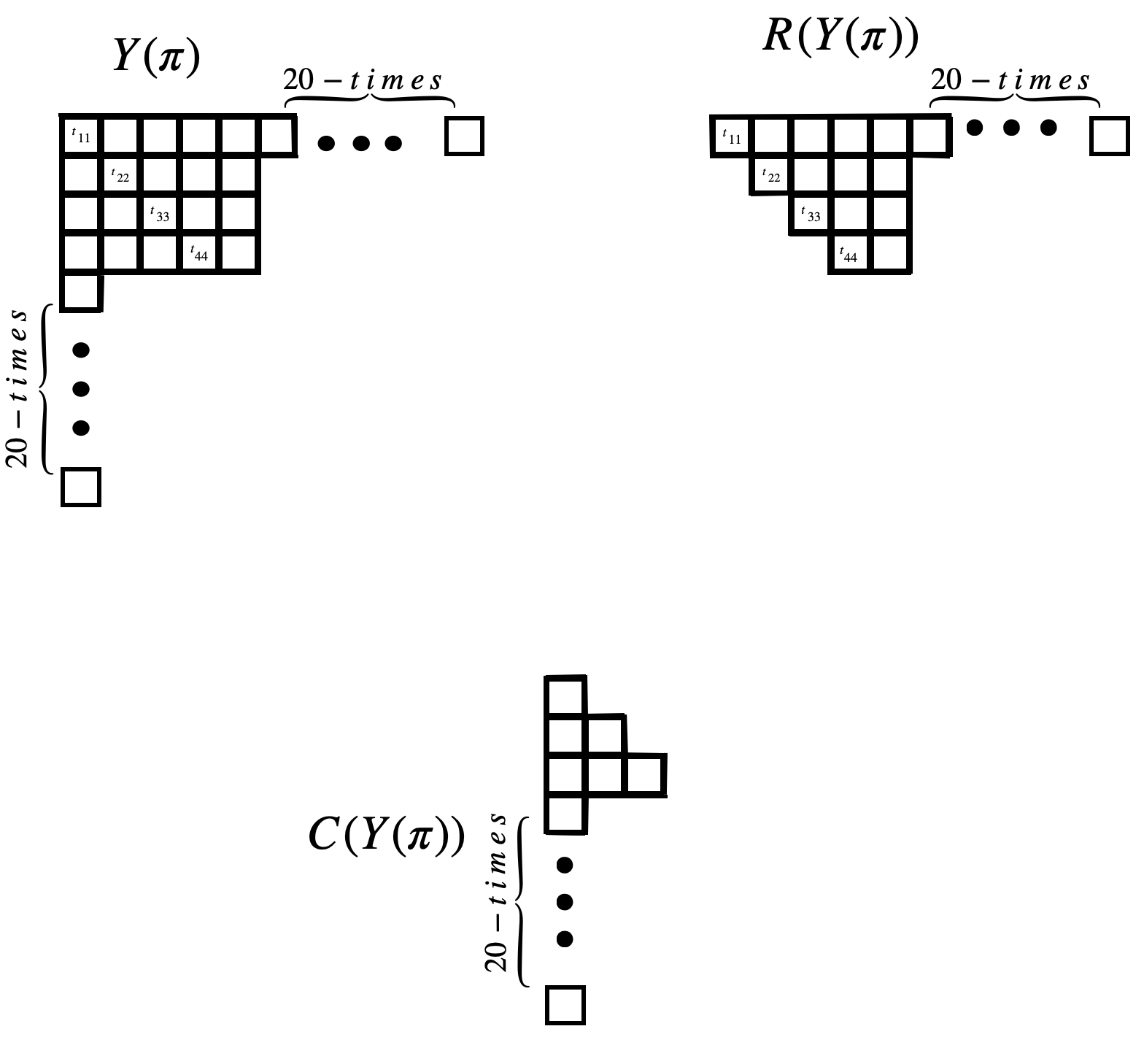}
 \end{center}
\caption{Division of $Y(\pi)$}
\label{5}
\end{figure}
If we look more closely at these shifted divisions of $Y(\pi)$. Each successive row of $R(Y(\pi))$ is shifted one block to the right. Furthermore, $R(Y(\pi))$ corresponding to sub-partition of $\pi$ forms a strictly decreasing sequence, that is, terms of of the sub-partition are distinct and these sub-partitions with distinct terms are called \textit{strict threshold eigen vales partitions}. Thus, if $\pi' = (a_1, a_2, \cdots, a_n)$ is a strict threshold eigen vales partition of a threshold eigenvalues partition $\pi$, then there is a unique shifted division whose $i^{th}$ row contains $a_i$ blocks, where $1 \leq i \leq n$. It follows that there is a one to one correspondence between the set of all threshold eigenvalue partitions of members of $\mathcal{T}_{p_1 \cdots p_n}$ and the set of all threshold eigen vales partition. As a result, $\mathcal{Y}$ is identical to the lattice, which we call \textit{lattice of shifted divisions}.

Recall that a subset $A$ of a poset $P$ is a \textit{chain} if any two elements of $A$ are comparable in $P$. A chain is called \textit{saturated} if there do not exist $x, z \in A$ and $y \in P\setminus{A}$ such that $y$ lies in between $x$ and $z$. In a locally finite lattice, a chain $\{x_0, x_1, \cdots, x_n\}$ of length $n$ is saturated if and only if $x_i$ covers $x_{i-1}$, where $1 \leq i \leq n$.

Since $\mathcal{T}_{p_1 \cdots p_n}$ is a locally finite distributive lattice, therefore $\mathcal{T}_{p_1 \cdots p_n}$ has a unique \textit{rank function} $\Psi : \mathcal{T}_{p_1 \cdots p_n} \longrightarrow \mathbb{Z}_{>0}$, where $\Psi\left(\Gamma(\mathbb{Z}/p_{r}^{\lambda_{r, 1}}\mathbb{Z}), \cdots, \Gamma(\mathbb{Z}/p_{r}^{\lambda_{r, n}}\mathbb{Z})\right)$ is the length of any saturated chain from $\hat{0}$ to the graph realized by a finite abelian $p_r$-group $\mathbb{Z}/p_{r}^{\lambda_{r,n}}\mathbb{Z}$. Note that a finite abelian $p_r$-group of rank $n$, $G_{\lambda_r, p_r} = \mathbb{Z}/p_{r}^{\lambda_{r, 1}}\mathbb{Z} \oplus \mathbb{Z}/p_{r}^{\lambda_{r, 2}}\mathbb{Z} \oplus \cdots \oplus \mathbb{Z}/p_{r}^{\lambda_{r,n}}\mathbb{Z}$ is identified with a sequence of abelian $p_r$ groups of rank $1$ $\left(\mathbb{Z}/p_{r}^{\lambda_{r, 1}}\mathbb{Z}, \mathbb{Z}/p_{r}^{\lambda_{r, 2}}\mathbb{Z}, \cdots, \mathbb{Z}/p_{r}^{\lambda_{r, n}}\mathbb{Z}\right)$ which in turn is identified with a sequence of graphs $\left(\Gamma(\mathbb{Z}/p_{r}^{\lambda_{r, 1}}\mathbb{Z}), \Gamma(\mathbb{Z}/p_{r}^{\lambda_{r, 2}}\mathbb{Z}), \cdots, \Gamma(\mathbb{Z}/p_{r}^{\lambda_{r, n}}\mathbb{Z})\right)$ or a sequence of a threshold partitions $(\mu_1, \mu_2, \cdots, \mu_n)\in \mathcal{Y}$. Therefore, the correspondence of $G_{\lambda_r, p_r} = \mathbb{Z}/p_{r}^{\lambda_{r, 1}}\mathbb{Z} \oplus \mathbb{Z}/p_{r}^{\lambda_{r, 2}}\mathbb{Z} \oplus \cdots \oplus \mathbb{Z}/p_{r}^{\lambda_{r,n}}\mathbb{Z}$ to $(\mu_1, \mu_2, \cdots, \mu_n)$ establishes that every finite abelain $p_r$-group of rank $n$ can be identified with a saturated chain in $\mathcal{T}_{p_1 \cdots p_n}$ or $\mathcal{Y}$ and the rank function of each abelian $p_r$-group of rank $n$ is $\Psi(\mu_1, \mu_2, \cdots, \mu_n) = {\lambda_{r,n}} =$ $max \{\lambda_{r,i} : 1\leq i \leq n \}$.
\begin{rem}
Let $\Lambda_{q}$ be the set of all non-isomorphic graphs of $\mathcal{T}_{p_1 \cdots p_n}$ with equal number of edges say $q$, (graphs realized by groups $\mathbb{Z}/2^{3}\mathbb{Z}$ and   $\mathbb{Z}/3^{2}\mathbb{Z}$ are non-isomorphic graphs with equal number of edges). Since there is one to one correspondence between threshold eigenvalues partitions and strict threshold eigenvalues partitions. The \textit{rank generating function} of the poset is presented in the following equation,

\begin{equation*}
\sum\limits_{q \geq 0}\kappa_{q}z^{q} = \prod\limits_{t \geq 1}(1 + z^{t}) = 1 + z + z^2 + 2z^3 + 2z^4 + \cdots,
\end{equation*}
where $\kappa_q$ is the cardinality of $\Lambda_{q}$.  
\end{rem}

The representation of a locally finite distributive lattice $\mathcal{T}_{235}$ is illustrated in Figure \eqref{6}. 
\begin{figure}[H]
 \begin{center}
 \includegraphics[scale=.415]{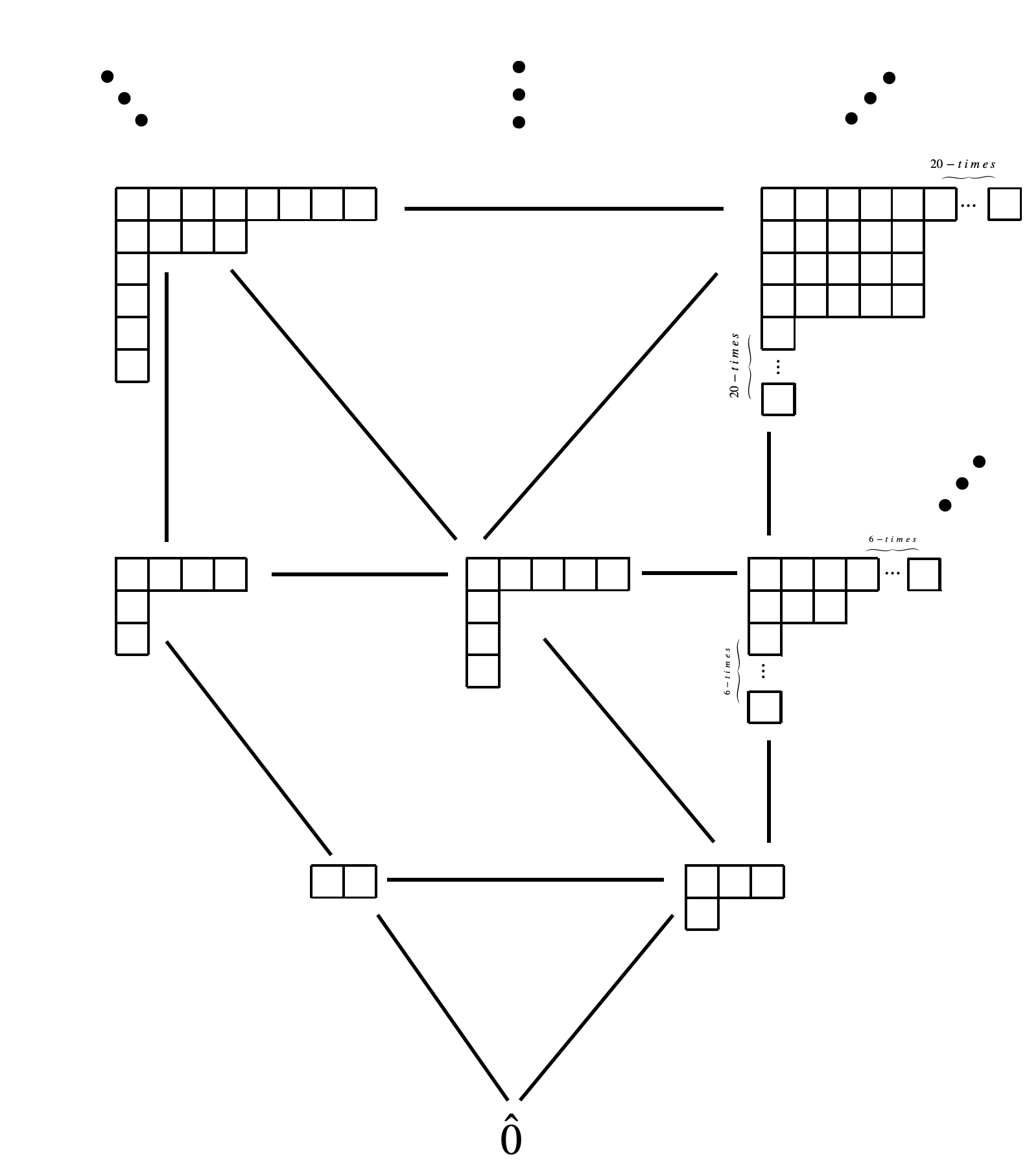}
 \end{center}
\caption{$\mathcal{T}_{235}$}
\label{6}
\end{figure}

Fix a finite abelian $p_r$-group $G$ and $\Gamma(G) \in \mathcal{T}_{p_1 \cdots p_n}$. Let $\ell(\Gamma(G))$ be the number of saturated chains in $\mathcal{T}_{p_1 \cdots p_n}$ from $\hat{0}$ to $\Gamma(G)$. 

The following result relates the number of saturated chains in $\mathcal{T}_{p_1 \cdots p_n}$ with the degree of a projective representation of a symmetric group $\mathcal{S}_t$ on $t$ number of symbols.

\begin{cor}
Let $\pi = (\pi_1, \pi_2, \cdots, \pi_k)$ be a strict threshold eigenvalues partition of some $\Gamma(G) \in \mathcal{T}_{p_1 \cdots p_n}$. Then the following hold,
\begin{equation}\label{eq3}
\ell(\Gamma(G)) = \frac{t!}{\prod\limits_{i =1}^{tr(Y(\pi))}\lambda_i!}\prod\limits_{r<s}\frac{\lambda_r - \lambda_s}{\lambda_r + \lambda_s},
\end{equation}
where $\lambda_i = \pi_{i} - i$, $1 \leq i \leq tr(Y(\pi))$ and $\lambda = (\lambda_1, \lambda_2, \cdots, \lambda_{tr(Y(\pi))})$ is a partition of some $t \in \Z_{>0}$.
\end{cor}
\begin{proof}
The right side of \eqref{eq3} represents the count of number of saturated chains from  $\hat{0}$ to $\Gamma(G)$.
\end{proof}

Note that the number of saturated chains  from $\hat{0}$ to $\Gamma(G)$ in \eqref{eq3} also provide a combinatorial formula for the number of finite abelian groups of different orders.

{\bf Acknowledgement:} This research project was initiated when the author visited School of Mathematics, TIFR Mumbai, India. So, I am immensely grateful to TIFR Mumbai for all the facilities. Moreover, I would like to thank Amitava Bhattacharya of TIFR Mumbai for some useful discussions on this research work.

\textbf{Declaration of competing interest.}\\
There is no conflict of interest to declare.

\textbf{Data Availability.}\\
Data sharing not applicable to this article as no datasets were generated or analysed during the current study.

\end{document}